\documentclass[11pt]{amsart}
\usepackage{amsopn}
\usepackage{amssymb, amscd}
\usepackage{paracol}
\usepackage{graphicx, graphics, epsfig}
\usepackage{faktor}
\usepackage{enumerate}
\usepackage{hyperref}
\usepackage{tikz-cd}
\globalcounter{figure}

\newcommand{\nc}{\newcommand}

\nc{\fg}{\mathfrak{f} } \nc{\vg}{\mathfrak{v} } \nc{\wg}{\mathfrak{w} }
\nc{\zg}{\mathfrak{z} } \nc{\ngo}{\mathfrak{n} } \nc{\kg}{\mathfrak{k} }
\nc{\mg}{\mathfrak{m} } \nc{\bg}{\mathfrak{b} } \nc{\ggo}{\mathfrak{g} }
\nc{\ggob}{\overline{\mathfrak{g}} } \nc{\sog}{\mathfrak{so} }
\nc{\sug}{\mathfrak{su} } \nc{\spg}{\mathfrak{sp} } \nc{\slg}{\mathfrak{sl} }
\nc{\glg}{\mathfrak{gl} } \nc{\cg}{\mathfrak{c} } \nc{\rg}{\mathfrak{r} }
\nc{\hg}{\mathfrak{h} } \nc{\tgo}{\mathfrak{t} } \nc{\ug}{\mathfrak{u} }
\nc{\dg}{\mathfrak{d} } \nc{\ag}{\mathfrak{a} } \nc{\pg}{\mathfrak{p} }
\nc{\sg}{\mathfrak{s} } \nc{\affg}{\mathfrak{aff} } \nc{\qg}{\mathfrak{q} }
\nc{\lgo}{\mathfrak{l} }

\nc{\pca}{\mathcal{P}} \nc{\nca}{\mathcal{N}} \nc{\lca}{\mathcal{L}}
\nc{\oca}{\mathcal{O}} \nc{\mca}{\mathcal{M}} \nc{\tca}{\mathcal{T}}
\nc{\aca}{\mathcal{A}} \nc{\cca}{\mathcal{C}} \nc{\gca}{\mathcal{G}}
\nc{\sca}{\mathcal{S}} \nc{\hca}{\mathcal{H}} \nc{\bca}{\mathcal{B}}
\nc{\dca}{\mathcal{D}} \nc{\val}{\operatorname{val}}

\nc{\vp}{\varphi} \nc{\ddt}{\frac{d}{dt}} \nc{\dds}{\frac{d}{ds}}
\nc{\dpar}{\frac{\partial}{\partial t}} \nc{\im}{\mathrm{i}}

\nc{\SO}{\mathrm{SO}} \nc{\Spe}{\mathrm{Sp}} \nc{\Sl}{\mathrm{SL}}
\nc{\SU}{\mathrm{SU}} \nc{\Or}{\mathrm{O}} \nc{\U}{\mathrm{U}} \nc{\Gl}{\mathrm{GL}}
\nc{\Se}{\mathrm{S}} \nc{\Cl}{\mathrm{Cl}} \nc{\Spein}{\mathrm{Spin}}
\nc{\Pin}{\mathrm{Pin}} \nc{\G}{\mathrm{GL}_n(\RR)} \nc{\g}{\mathfrak{gl}_n(\RR)}

\nc{\RR}{{\Bbb R}} \nc{\HH}{{\Bbb H}} \nc{\CC}{{\Bbb C}} \nc{\ZZ}{{\Bbb Z}}
\nc{\FF}{{\Bbb F}} \nc{\NN}{{\Bbb N}} \nc{\QQ}{{\Bbb Q}} \nc{\PP}{{\Bbb P}} \nc{\OO}{{\Bbb O}}

\nc{\vs}{\vspace{.2cm}} \nc{\vsp}{\vspace{1cm}} \nc{\ip}{\langle\cdot,\cdot\rangle}
\nc{\ipp}{(\cdot,\cdot)} \nc{\la}{\langle} \nc{\ra}{\rangle} \nc{\unm}{\tfrac{1}{2}}
\nc{\unc}{\tfrac{1}{4}} \nc{\und}{\tfrac{1}{16}} \nc{\no}{\vs\noindent}
\nc{\lam}{\Lambda^2(\RR^n)^*\otimes\RR^n} \nc{\tangz}{{\rm T}^{\rm Zar}}
\nc{\nor}{{\sf n}}  \nc{\mum}{/\!\!/} \nc{\kir}{/\!\!/\!\!/}
\nc{\Ri}{\tfrac{4\Ric_{\mu}}{||\mu||^2}} \nc{\ds}{\displaystyle}
\nc{\ben}{\begin{enumerate}} \nc{\een}{\end{enumerate}} \nc{\f}{\frac}
\nc{\lb}{[\cdot,\cdot]} \nc{\isn}{\tfrac{1}{||v||^2}}
\nc{\gkp}{(\ggo=\kg\oplus\pg,\ip)} \nc{\ukh}{(\ug=\kg\oplus\hg,\ip)}
\nc{\tgkp}{(\tilde{\ggo}=\kg\oplus\pg,\ip)}
\nc{\wt}{\widetilde}
\nc{\iop}{\mathtt{i}} \nc{\jop}{\mathtt{j}} \nc{\gk}{g_{\kil}}

\nc{\alp}{\alpha^2}  \nc{\bet}{\beta^2}  \nc{\gam}{\gamma^2} \nc{\et}{\eta^2}

\nc{\Hess}{\operatorname{Hess}} \nc{\ad}{\operatorname{ad}}
\nc{\Ad}{\operatorname{Ad}} \nc{\rank}{\operatorname{rank}}
\nc{\Irr}{\operatorname{Irr}} \nc{\End}{\operatorname{End}}
\nc{\Aut}{\operatorname{Aut}} \nc{\Inn}{\operatorname{Inn}}
\nc{\Der}{\operatorname{Der}} \nc{\Ker}{\operatorname{Ker}}
\nc{\Iso}{\operatorname{Iso}} \nc{\Diff}{\operatorname{Diff}}
\nc{\Lie}{\operatorname{L}} \nc{\tr}{\operatorname{tr}} \nc{\dif}{\operatorname{d}}
\nc{\sen}{\operatorname{sen}} \nc{\modu}{\operatorname{mod}}
\nc{\CRic}{\operatorname{PP}} \nc{\Cric}{\operatorname{P}} \nc{\Ricci}{\operatorname{Ric}}
\nc{\sym}{\operatorname{sym}} \nc{\herm}{\operatorname{herm}} \nc{\symac}{\operatorname{sym^{ac}}}
\nc{\symc}{\operatorname{sym^{c}}} \nc{\scalar}{\operatorname{Sc}}
\nc{\grad}{\operatorname{grad}} \nc{\ricci}{\operatorname{Rc}}
\nc{\Nor}{\operatorname{Norm}}  \nc{\ricc}{\operatorname{Rc^{c}}}
\nc{\Ricc}{\operatorname{Ric^{c}}} \nc{\ricac}{\operatorname{Rc^{ac}}}
\nc{\Ricac}{\operatorname{Ric^{ac}}} \nc{\Riem}{\operatorname{Rm}} \nc{\Sec}{\operatorname{Sec}}
\nc{\riccig}{\operatorname{ric^{\gamma}}} \nc{\Rin}{\operatorname{M}}
\nc{\kil}{\operatorname{B}} \nc{\cas}{\operatorname{C}} \nc{\Le}{\operatorname{L}}
\nc{\tang}{\operatorname{T}}
\nc{\level}{\operatorname{level}} \nc{\rad}{\operatorname{r}}
\nc{\abel}{\operatorname{ab}} \nc{\CH}{\operatorname{CH}} \nc{\Cone}{{\mathcal C}} \nc{\CCone}{\operatorname{CC}} \nc{\CP}{{\mathcal P}}
\nc{\mcc}{\operatorname{mcc}} \nc{\Adj}{\operatorname{Adj}}
\nc{\Order}{\operatorname{O}}  \nc{\inj}{\operatorname{inj}} \nc{\proy}{\operatorname{pr}}
\nc{\vol}{\operatorname{vol}} \nc{\Diag}{\operatorname{Dg}} \nc{\Diagg}{\operatorname{Diag}}
\nc{\Spec}{\operatorname{Spec}} \nc{\Ima}{\operatorname{Im}} \nc{\Rea}{\operatorname{Re}}
\nc{\spann}{\operatorname{span}} \nc{\Aff}{\operatorname{Aff}}
\nc{\mm}{\operatorname{m}} \nc{\id}{\operatorname{Id}} \nc{\Bb}{\operatorname{B}}
\nc{\lic}{\operatorname{L}}

\newtheorem{theorem}{Theorem}[section]
\newtheorem{proposition}[theorem]{Proposition}

\newtheorem{lemma}[theorem]{Lemma}

\theoremstyle{definition}

\theoremstyle{remark}
\newtheorem{remark}[theorem]{Remark}

\title[Stability of standard Einstein metrics]{Stability of standard Einstein metrics on homogeneous spaces of non-simple Lie groups}

\author{Valeria Guti\'errez and Jorge Lauret}

\address{FaMAF, Universidad Nacional de C\'ordoba and CIEM, CONICET (Argentina)}
\email{valeria.gutierrez@unc.edu.ar, jorgelauret@unc.edu.ar}

\thanks{This research was partially supported by grants from FONCyT and Univ.\ Nac.\ de C\'ordoba.}

\date{\today}

\begin{document}

\begin{abstract}
The classification of compact homogeneous spaces of the form $M=G/K$, where $G$ is a non-simple Lie group, such that the standard metric is Einstein is still open.  The only known examples are
$4$ infinite families and $3$ isolated spaces found by Nikonorov and Rodionov in the 90s.  In this paper, we prove that most of these standard Einstein metrics are unstable as critical points of the scalar curvature functional on the manifold of all unit volume $G$-invariant metrics on $M$, providing a lower bound for the coindex in the case of Ledger-Obata spaces.  On the other hand, examples of stable (in particular, local maxima) invariant Einstein metrics on certain homogeneous spaces of non-simple Lie groups are also given.  
\end{abstract}

 \maketitle

 \tableofcontents

\section{Introduction}\label{intro}

Any manifold $M$ on which a given compact semisimple Lie group $G$ is acting transitively can be endowed with a canonical $G$-invariant Riemannian metric $g_{\kil}$ provided by the Killing form of the Lie algebra $\ggo$ of $G$, so called the {\it standard} metric.  In the case when $G$ is simple, the Einstein condition for $g_{\kil}$ was studied by Wang and Ziller in \cite{WngZll2}, where it is obtained a complete classification.  The list of homogeneous spaces $M=G/K$ with $G$ simple, beyond isotropy irreducible spaces, on which $\gk$ is Einstein consists of $10$ infinite families and $2$ isolated examples for $G$ classical and $20$ isolated examples with $G$ exceptional.  The stability types of these metrics as critical points of the scalar curvature functional 
$$
\scalar:\mca_1^G\longrightarrow \RR,
$$ 
where $\mca^G_1$ is the manifold of all $G$-invariant metrics of some fixed volume on $M$, have recently been obtained in \cite{stab}.  The standard metric is $G$-unstable on most of the spaces with $G$ classical, being even a local minimum on many of them (see \cite[Table 1]{stab}), while $g_{\kil}$ is $G$-stable and therefore a local maximum on $14$ of the $20$ spaces with exceptional $G$ (see \cite[Tables 2,3,4]{stab}).

On the other hand, in the case when $G$ is not simple, the classification of standard Einstein metrics is still an open question (see \cite[Section 4.14]{NknRdnSlv} and references therein).  Under certain conditions, Nikonorov obtained in \cite{Nkn} the following algebraic constraints on the structure of homogeneous spaces on which the standard metric $\gk$ is Einstein.  We fix a decomposition 
\begin{equation}\label{dec-intro}
\ggo=\ggo_1 \oplus \ggo_2 \oplus \cdots \oplus \ggo_m,
\end{equation}
of the semisimple Lie algebra $\ggo$ in simple ideals and consider $\pi_j:\kg \rightarrow \ggo_j$, the corresponding projections.

\begin{theorem}\label{Nthm}\cite[Theorem 6]{Nkn}.
Let $(G/K,\gk)$ be a connected irreducible standard homogeneous Einstein space with semisimple isotropy group $K$ such that either $\pi_j(\kg)=\ggo_j$ or $\pi_j(\kg)$ consists of two simple summands. Then it is defined by the following scheme of inclusion:
\begin{equation*}
  \begin{aligned}
  K:=H \times L_1 \times \cdots \times L_n &\subset\underbrace{ H \times \cdots \times H}_m \times L_1 \times \cdots \times L_n \\
  &\subset \underbrace{H\times \cdots \times H}_{m-n} \times G_1 \times \cdots \times G_n =G,
  \end{aligned}
\end{equation*}
where $H, \ L_i$ and $G_i$ are simple Lie groups, the first inclusion has the form $\operatorname{diag}(H)\times \id \times \cdots \times \id$, the second has the form $\id \times \cdots \times \id \times \iota_1 \times \cdots \times\iota_n$, where $\iota_i: H\times L_i \rightarrow G_i$ are some inclusions.
\end{theorem}

This is a strong obstruction to have $\gk$ Einstein, in particular, it implies that the simple factors of $G$ on which the projection of $K$ is onto are pairwise isomorphic.  In \cite[Theorem 7]{Nkn}, a complete list of homogeneous spaces of the above form on which $\gk$ is Einstein is obtained, consisting of $4$ infinite families and $3$ isolated examples. 

We study in this paper the stability of these Einstein metrics.  Our main result is the following. 
 
\begin{theorem} 
Let $M=G/K$ be a homogeneous space as in Theorem \ref{Nthm} and assume that the standard metric $\gk$ is Einstein and $n+4\leq m$.  
\begin{enumerate}[{\rm (i)}]
\item $\gk$ is $G$-unstable, i.e., there is a positive direction for the Hessian of $\scalar$ at $\gk$ (see Theorem \ref{N-stab}).

\item On the Ledger-Obata space $M=H^{m}/\Delta^{m}H$ (i.e., when $n=0$ and $m\geq 3$), $\gk$ has coindex $\geq m-3$ if $12\leq m$ and coindex $\geq m-2$ if $3\leq m\leq 11$ (see Theorem \ref{LO}).  
\end{enumerate}
\end{theorem}

It is worth noting that we do not use the classification given in \cite[Theorem 7]{Nkn} to prove part (i), we only use the structure provided by Theorem \ref{Nthm} and the general formulas for the Hessian of $\scalar$ given in \cite{stab-tres} and \cite{stab-dos}.  Most cases in \cite[Theorem 7]{Nkn} are covered by the above theorem (see Remark \ref{N-stab-rem}).  

The only other stability result on $G$-invariant Einstein metrics with a non-simple $G$ we were able to find in the literature is in \cite[Section 7]{stab-tres}, where it is proved that the so called Jensen's metric on a simple Lie group $H$ (see \cite{Jns}) is always $G$-unstable viewed as a $G$-invariant metric on $M=G/\Delta K$, where $G:=H\times K$, for some semisimple subgroup $K\subset H$.  The following natural question arises: is any $G$-invariant Einstein metric on a homogeneous space $G/K$ with $G$ non-simple $G$-unstable?  The following result answers this question in the negative.    

\begin{theorem} (See Theorem \ref{prod2}).  
For any simple Lie group $H$ and simple subgroup $K\subset H$ such that the homogeneous space $H/K$ is isotropy irreducible, 
there exists an $(H\times K)$-invariant Einstein metric on $M=H\times K/\Delta K$ which is $(H\times K)$-stable.
\end{theorem}

This $(H\times K)$-stable Einstein metric is therefore a local maximum of $\scalar:\mca_1^G\longrightarrow \RR$ and it is different from the Jensen's metric mentioned above, which is a local minimum (see Remark \ref{Jensen-rem}).  

\vs \noindent {\it Acknowledgements.}  The authors thank Yuri Nikonorov for very helpful conversations on the subject of this paper and the referee for very helpful suggestions.

\section{Preliminaries}\label{pre}

We fix an almost-effective transitive action of a compact connected Lie group $G$ on a manifold $M^d$, determining a presentation $M=G/K$ of $M$ as a homogeneous space, where $K\subset G$ is the isotropy subgroup at some point $o\in M$.  Let $\mca^G$ denote the manifold of all $G$-invariant metrics on $M$, which satisfies $1\leq\dim{\mca^G}\leq\tfrac{d(d+1)}{2}$.

If $\mca^G_1\subset\mca^G$ is the codimension one submanifold of all unit volume metrics, then $g\in\mca_1^G$ is Einstein if and only if $g$ is a critical point of the scalar curvature functional $\scalar:\mca_1^G\longrightarrow \RR$ and the stability type of $g$ is therefore encoded in the signature of the second derivative or Hessian $\scalar''_g$.  An Einstein metric $g\in\mca_1^G$ with $\ricci(g)=\rho g$ is said to be {\it $G$-{\it stable}} when $\scalar''_g|_{\tca\tca_g^G}<0$, where $\tca\tca_g^G$ is the space of all $G$-invariant TT-tensors, which in particular implies that $g$ is a local maximum of $\scalar|_{\mca_1^G}$.  On the other hand, it is called {\it $G$-{\it unstable}} when there exists $T\in\tca\tca_g^G$ such that $\scalar''_g(T,T)>0$, being the \emph{coindex} the dimension of the maximal subspace of $\tca\tca_g^G$ on which $\scalar''_g$ is positive (see \cite[Section 3]{stab-tres}, \cite[Section 2]{stab-dos} and \cite[Section 2]{stab} for more detailed treatments on $G$-stability).

We assume from now on that $G$ is semisimple and consider the {\it standard} metric $\gk$, i.e., the $G$-invariant metric determined by the inner product 
$$
\ip:=-\kil_\ggo|_{\pg\times\pg}, 
$$ 
where $\kil_\ggo$ is the Killing form of $\ggo$ and $\ggo=\kg\oplus\pg$ is the $\kil_\ggo$-orthogonal reductive decomposition.  Here
$\ggo$ and $\kg$ are the Lie algebras of $G$ and $K$, respectively.  Consider the vector space
$$
\sym_0(\pg)^K:=\{A:\pg\rightarrow\pg: A^t=A, \ \tr{A}=0, \ [\Ad(K),A]=0\} \equiv \tca\tca_{\gk}^G,
$$
where $A^t$ denotes the transpose of $A$ with respect to $\gk$ (see \cite[Section 2]{stab}).

We also assume that $\gk$ is Einstein, say with $\ricci(g)=\rho g$.  The second variation of the scalar curvature at $\gk\in\mca^G$ is given by
\begin{equation}\label{ScLL}
\scalar''_{\gk}(T,T) = \unm(2\rho \la A,A\ra - \la\lic_\pg A,A\ra), \qquad\forall T=\gk(A\cdot,\cdot)\in \tca\tca_{\gk}^G, \quad A\in\sym_0(\pg)^K,
\end{equation}
where $\la A,B\ra:=\tr{AB}$,
$$
\lic_\pg=\lic_\pg(\gk):\sym_0(\pg)^K\longrightarrow\sym_0(\pg)^K,
$$
is the self-adjoint operator defined by
\begin{equation}\label{Lpnr-intro}
\lic_\pg A:=-\unm\sum[\ad_\pg{X_i},[\ad_\pg{X_i},A]], \qquad\forall A\in\sym_0(\pg)^K,
\end{equation}
and $\{ X_i\}$ is any $\gk$-orthonormal basis of $\pg$ (see \cite[Section 5]{stab-tres}).

It follows from \eqref{ScLL} that the $G$-stability type of $\gk$ is determined by how is the constant $2\rho$ suited relative to the spectrum of $\lic_\pg$.  Indeed, if $\lambda_\pg$ is the minimum eigenvalue of $\lic_\pg$, then
\begin{enumerate}[{\small $\bullet$}]
\item $\gk$ is $G$-stable if and only if $2\rho<\lambda_\pg$;

\item $\gk$ is $G$-unstable if and only if $\lambda_\pg<2\rho$.
\end{enumerate}

Given any orthogonal decomposition $\pg=\pg_1\oplus\dots\oplus\pg_r$ in $\Ad(K)$-invariant (not necessarily irreducible) subspaces $\pg_1,\dots,\pg_r$ ($d_i:=\dim{\pg_i}$), consider the corresponding structural constants given by,
\begin{equation}\label{SC}
[ijk]:=\sum_{\alpha,\beta,\gamma} \la[X_\alpha^i,X_\beta^j], X_\gamma^k\ra^2,
\end{equation}
where $\{ X_\alpha^i\}$ is an orthonormal basis of $\pg_i$.  Note that the number $[ijk]$ is invariant under any permutation of $ijk$.  According to \cite[Theorem 5.3]{stab-tres} (see also \cite[Theorem 3.1]{stab-dos}), if $\left\{ I_1,\dots, I_r\right\}$ is the orthonormal subset of $\sym(\pg)^K:=\RR I_\pg\oplus\sym_0(\pg)^K$ defined by $I_k|_{\pg_i}:=\delta_{ki}\tfrac{1}{\sqrt{d_k}}I_{\pg_k}$ ($I_\vg$ denotes the identity map on the vector space $\vg$), then
\begin{equation}\label{Lpijk}
\begin{aligned}
\la\lic_\pg I_k, I_k\ra &= \tfrac{1}{d_k}\sum\limits_{j\ne k;i} [ijk], \qquad\forall k, \\
\la\lic_\pg I_k, I_m\ra &=  -\tfrac{1}{\sqrt{d_k}\sqrt{d_m}}\sum\limits_{i} [ikm], \qquad\forall k\ne m.
\end{aligned}
\end{equation}
In the \emph{multiplicity-free} case (i.e., the subspaces $\pg_k$'s are $\Ad(K)$-irreducible and pairwise inequivalent), $\left\{ I_1,\dots, I_r\right\}$ is actually an orthonormal basis of $\sym(\pg)^K$ and so the above numbers are precisely the entries of the matrix of $\lic_\pg$.  In any case, the spectrum of the symmetric $r\times r$ matrix $[\la\lic_\pg I_k, I_m\ra]$ defined as in \eqref{Lpijk} (restricted to the hyperplane $\{(a_1,\dots,a_s):\sum d_ia_i=0\}$) is still contained in $[\lambda_\pg, \lambda_\pg^{\max}]$, where $\lambda_\pg^{max}$ is the maximum eigenvalue of $\lic_\pg$, so this always provides a very useful tool to compute or estimate $\lambda_\pg$.

\section{G-instability of standard Einstein metrics}\label{nik-sec}

Let $M=G/K$ be a homogeneous space as in Theorem \ref{Nthm}.  Thus
\begin{equation}\label{decgk}
\ggo=\underbrace{\hg\oplus\dots\oplus\hg}_{m-n} \oplus\ggo_1\oplus\dots\oplus\ggo_n, \qquad 
\kg=\hg\oplus\lgo_1\oplus\dots\oplus\lgo_n,
\end{equation}
and $\pi_{j}(\kg)=\iota_j(\hg\times \lgo_j)\subset\ggo_j$ for all $j=1,\dots,n$ (note that we are using the index $j$ instead of $m-n+j$ for simplicity).  For each homogeneous space $G_j/\pi_j(K)$, we consider the $\Bb_{\ggo_j}$-orthogonal reductive decomposition
$$
\ggo_j= \iota_j(\hg \oplus \lgo_j) \oplus \mathfrak{q}_j, \qquad j=1, \dots, n.
$$
Since $\pi_j(\hg)$ is a subalgebra of $\ggo_j$, there exists $0 \leq c_j \leq 1$ such that
$$
\Bb_{\pi_j(\hg)} = c_j\Bb_{\ggo_j}|_{\pi_j(\hg)}, \qquad j=1, \dots, n.
$$
Note that $c_j=0$ if and only if $\pi_j(\hg)$ is abelian and $c_j=1$ if and only if $\pi_j(\hg)=\ggo_j$.

It is easy to check that the $\Bb_{\ggo}$-orthogonal reductive complement for $G/K$ is given by
  \begin{align}
  \pg:=\{&(X_1,\ldots, X_{m-n},\pi_1(X_{m-n+1}) + Y_1,\ldots,\pi_n(X_m)+Y_n) \notag \\
  & : \ X_i \in \hg, \ Y_j \in \mathfrak{q}_j \
  \text{and} \ X_1+\ldots+X_{m-n}+\tfrac{1}{c_1}X_{m-n+1}+\ldots+\tfrac{1}{c_n}X_m=0 \}. \label{red}
  \end{align}
In particular, $\dim{M}=\dim{\pg}=(m-1)\dim{\hg} + \dim{\qg_1}+\dots+\dim{\qg_n}$.

Given  $\alpha:=(a_1,\ldots,a_{m-n},\ldots, a_m)\in\RR^m$ such that $\sum a_i =0$, we define an $\Ad(K)$-irreducible subspace of $\pg$ by,
$$
\ggo_\alpha := \left\{(a_1X,\ldots,a_{m-n}X, c_1a_{m-n+1}\pi_1(X),\ldots, c_na_m\pi_n(X))\ :\ X\in \hg \right\},$$
and if
\begin{equation}\label{alfai}
\alpha_i:=(\overbrace{1,\ldots,1}^i,-i,\overbrace{0,\ldots,0}^{m-1-i}), \qquad i=1,\dots,m-1,
\end{equation}
then we consider the following $\gk$-orthogonal decomposition of $\pg$ in $\Ad(K)$-invariant subspaces:
$$
\pg=\ggo_{\alpha_1}\oplus \ggo_{\alpha_2} \oplus \cdots \oplus \ggo_{\alpha_{m-1}}\oplus \mathfrak{q}_1\oplus \cdots \oplus \mathfrak{q}_n,
$$
where $\qg_j$ also denotes the subspace $(0,\dots,0,\qg_j,0,\dots,0)\subset\ggo$.  Note that $\dim{\ggo_{\alpha_i}}=\dim{\hg}$ for all $i$.  

Concerning the corresponding structural constants $[ijk]$ defined in \eqref{SC} (set $\pg_i:=\ggo_{\alpha_i}$, $i=1,\dots,m-1$ and $\pg_{m+j}:=\qg_{j+1}$, $j=0,\dots,n-1$), we have that:
\begin{enumerate}[{\small $\bullet$}]
  \item $[\ggo_{\alpha_i},\ggo_{\alpha_k}]\subset\ggo_{\alpha_i}$ for all $i< k \leq m-n-1$, so $[iik]>0$ for all $i<k\leq m-n-1$.
  \item $[iik]$ is also positive when $i<m-n\leq k \leq m-1$.
  \item There are others positive structural constants involving the spaces $\mathfrak{q}_j$'s but we do not need them for our computations.
 \end{enumerate}

\begin{lemma}\label{str}
Let $M=G/K$ be a homogeneous space as in Theorem \ref{Nthm} such that $n<m$.
\begin{enumerate}[{\rm (i)}]
\item For any $1\leq i\neq k\leq m-n-1$ and $0\leq j\leq n-1$,
$$
[iik]= \frac{\dim{\hg}}{k+k^2}, \qquad [ii(m-n+j)]= \frac{\dim{\hg}}{\Delta_j},
$$
where $\Delta_j:=m-n+c_1+\ldots+c_j+c_{j+1}(m-n+j)^2$.

\item If $\gk$ is Einstein, say $\ricci(\gk)=\rho\gk$, then
$$
\rho= \frac{3}{4}-\frac{m-n-1}{2(m-n)}-\frac{1}{2}\sum_{j=0}^{n-1}\frac{1}{\Delta_j}.
$$
\end{enumerate}
\end{lemma}

\begin{proof}
Let $\{X_1,\ldots,X_{\dim{\hg}}\}$ be a $(-\Bb_{\hg})$-orthonormal basis of $\hg$. For each $X\in\hg$ we denote 
$$
\alpha(X):=(a_1X,\ldots,a_{m-n}X,c_1a_{m-n+1}\pi_1(X),\ldots,c_na_m\pi_n(X))\in\ggo_\alpha,
$$
where $\alpha=(a_1,\ldots,a_{m})$ and $\sum a_i=0$.  Thus for each $i=1,\ldots,m-n-1$, the set $\left\{\tfrac{1}{|\alpha_i|}\alpha_i(X_l)\right\}$ is a $(-\Bb_\ggo)$-orthonormal basis of $\pg_i$ (see \eqref{alfai}) and we obtain that
$$
  \left[iik\right]=  \sum_{l,m,r}\left<\left[\tfrac{\alpha_i(X_r)}{|\alpha_i|},\tfrac{\alpha_i(X_l)}{|\alpha_i|}\right],\tfrac{\alpha_k(X_m)}{|\alpha_k|}\right>^2 
  =  \left(\tfrac{\left<\alpha_i\cdot\alpha_i,\alpha_k\right>}{|\alpha_i|^2|\alpha_k|}\right)^2 \sum_{l,m,r}\left<[X_r,X_l],X_m\right>^2,   
$$
for all $i<k\leq m-n-1$. where $\alpha_i\cdot\alpha_i:=(1,1,\ldots 1, i^2,0,\ldots 0)$.  Since $\left<\cdot,\cdot \right>$ is the usual inner product in $\RR^m$, the formula for these structural constants follows.

For the cases $\pg_{m-n+j}$ with $j=0,\ldots,n-1$ we have that $\left\{\frac{1}{\sqrt{\Delta_j}}\alpha_{m-n+j}(X_l)\right\}$ is a $-\Bb_{\ggo}$-orthonormal basis of $\pg_{m-n+j}$ and so a computation similar to the above completes the proof of part (i).

To prove part (ii), we use the well known formula for the Ricci eigenvalues of $\gk$ in terms of the structural constants (see e.g.\ \cite{stab-tres} or \cite{stab-dos}) assuming that  $M=G/K$ is a standard homogeneous Einstein space, it is given by
 \begin{align*}
  \rho=&\rho_1=\frac{1}{2}-\frac{1}{4\dim{\hg}}\sum_{i,j}\left[ij1\right] \\ 
  =&\frac{1}{2}-\frac{1}{4\dim{\hg}}\left(2\sum^{m-1}_{k=2}\left[k11\right]\right) = \frac{1}{2}-\frac{1}{2\dim{\hg}}\left(\sum^{m-n-1}_{k=2}\left[11k\right]+\sum^{m-1}_{k=m-n}\left[11k\right]\right) \\
    =& \frac{1}{2}-\frac{1}{2\dim{\hg}}\left(\sum^{m-n-1}_{k=2}\frac{\dim{\hg}}{k+k^2}+\sum^{n-1}_{j=0}\left[11(m-n+j)\right]\right) \\
    =& \frac{1}{2}-\frac{1}{2}\left(\frac{m-n-1}{m-n}-\frac{1}{2}+\sum^{n-1}_{j=0}\frac{1}{\Delta_j}\right) = \frac{3}{4}-\frac{m-n-1}{2(m-n)}-\frac{1}{2}\sum_{j=0}^{n-1}\frac{1}{\Delta_j},
  \end{align*}
concluding the proof.
\end{proof}

\begin{theorem}\label{N-stab}
If $M=G/K$ is a homogeneous space as in Theorem \ref{Nthm} such that $n+4\leq m$ and the standard metric $\gk$ is Einstein, then $\gk$ is $G$-unstable.
\end{theorem}

\begin{remark}\label{N-stab-rem}
The spaces in \cite[Table 4]{Nkn} covered by the above theorem are items 3 ($s\geq 2$) and 6, as well as most cases in items 1 and 2 (see \cite[Section 2]{NknRdn}).  Thus the only missing spaces from \cite[Theorem 7]{Nkn} are the space in part 3) and items 4 and 5 in the table. 
\end{remark}

\begin{proof}
Consider, as in \S\ref{pre}, the orthonormal subset $\mathcal{C}=\{I_1,\ldots I_{m-n-1}\}$ of $\sym(\pg)^K$ given by $I_k|_{\pg_i}:=\delta_{ik}\frac{1}{\sqrt{\dim{\hg}}}I_{\pg_k}$ and $I_k|_{\qg_j}=0$ for all $j$. By the formula given in (\ref{Lpijk}), if $r,s=1,\ldots,m-n-1$, then 
\begin{equation*}
  \begin{aligned}
  \la\lic_\pg I_r, I_s\ra  = -\frac{1}{\dim{\hg}}\sum_{i}[irs]=&-\frac{1}{\dim{\hg}}([srs]+[rrs])
  = \left\{\begin{aligned}
  &\frac{-1}{s+s^2} \quad \text{ if } \ s>r, \\
  &\frac{-1}{r+r^2} \quad \text{ if } \ s<r, \\
  \end{aligned} \right. \\
  \la\lic_\pg I_s, I_s\ra =  \frac{1}{\dim{\hg}}\sum_{k\neq s;i}[iks]=&\frac{1}{\dim{\hg}}\sum_{i=1}^{s-1}[iis]+\frac{1}{\dim{\hg}}\sum_{k=s+1}^{m-1}[sks] \\
  =&\frac{1}{\dim{\hg}}\sum_{i=1}^{s-1}\frac{\dim \hg}{s+s^2}+\frac{1}{\dim{\hg}}\left(\sum_{k=s+1}^{m-n-1}\left[kss\right]+ \sum_{k=m-n}^{m-1}\left[kss\right]\right)\\
  =&\frac{1}{\dim{\hg}}\sum_{i=1}^{s-1}\frac{\dim \hg}{s+s^2}+\frac{1}{\dim{\hg}}\left(\sum_{k=s+1}^{m-n-1}\frac{\dim \hg}{k+k^2}+ \sum_{j=0}^{n-1}\left[ss(m-n+j)\right]\right)\\
  = &\frac{s-1-s^2}{s+s^2}+\underbrace{\frac{m-n-1}{m-n}+\sum_{j=0}^{n-1}\frac{1}{\Delta_j}}_{-2(\rho-\tfrac{3}{4})}\\
   =& \frac{s-1-s^2}{s+s^2}-2\rho+\frac{3}{2}.
  \end{aligned}
\end{equation*}
It is therefore easy to check that with respect to the basis $\mathcal{C}$, $A:=\tfrac{1}{\sqrt{6}}(1,1,-2,0,\ldots,0)$ is a unit direction such that,
 \begin{equation*}
 \lambda_\pg\leq \left< \lic_\pg A,A\right> = 1-2\rho < 2\rho,
\end{equation*}
which implies that $\gk$ is $G$-unstable.  The inequality $\unc<\rho$ is well known, see e.g.\ \cite[Section 3.1]{stab}.
\end{proof}

\section{Ledger-Obata spaces}\label{LO-sec}

In this section, given a connected compact simple Lie group $F$, we consider the so-called {\it Ledger-Obata} spaces
$$
M=G/K=F^{m+1}/\Delta^{m+1}F, \qquad 2\leq m,
$$
where $F^{m+1}:=F\times\dots \times F$ ($(m+1)$-times) and $\Delta^{m+1}F:=\{(x,\ldots,x):x\in F\}$.  The Lie group $G=F^{m+1}$ acts transitively on $\overline{G}:=F^m$ by
$$
(x_1,\dots, x_{m+1})\cdot(y_1,\dots, y_m)=(x_1y_1x_{m+1}^{-1},\ldots, x_my_mx_{m+1}^{-1}),
$$
with isotropy group at the identity given by $K=\Delta^{m+1}F$.  This is therefore a particular case of the family of spaces studied in \S\ref{nik-sec} (i.e., $n=0$ and $m+1$ instead of $m$).  The standard metric $\gk$ is always Einstein, say   $\ricci(g_B)=\rho g_B$, on any Ledger-Obata space (see \cite{ChnNknNkn}).  We aim to obtain a lower bound for the coindex of this $G$-unstable critical point.  

If $\fg $ denotes the Lie algebra of $F$, then the $\kil_\ggo$-orthogonal reductive decomposition $\ggo=\kg\oplus\pg$ is given by $\ggo=\fg^{m+1}$, $\kg=\Delta^{m+1}\fg$ and
$$
\pg:=\left\{(X_1,\ldots,X_{m+1}) : X_i\in \fg \text{ and } \sum X_i=0\right\}.
$$
Thus any Ad($K$)- irreducible subspace of $\pg$ has the form
$$
\ggo_\alpha := \left\{(a_1X,\ldots, a_{m+1}X):X\in \fg\right\},
$$
for some fixed $\alpha:=(a_1,\ldots, a_{m+1})$ such that $\sum a_i =0$. Note that $B_\ggo(\ggo_\alpha,\ggo_\beta)=0$ if and only if $\alpha \ \bot \ \beta$ and
$[\ggo_\alpha,\ggo_\beta]=\ggo_{\alpha\cdot\beta}$, where $\alpha\cdot\beta=(a_1b_1,\ldots ,a_{m+1}b_{m+1})$.  As before, we consider an orthogonal decomposition $\pg=\ggo_{\alpha_1}\oplus \ggo_{\alpha_2} \oplus \cdots \oplus \ggo_{\alpha_m}$ in Ad($K$)-invariant subspaces, where
$$
\alpha_i=(\underbrace{1,\ldots,1}_i,-i,\underbrace{0,\ldots,0}_{m-i})\in\RR^{m+1}.
$$
Note that $\dim{\ggo_{\alpha_i}}=\dim{\fg}$ for all $i=1,\dots, m$.

We consider the orthonormal subset $\mathcal{C}=\{I_1,\ldots I_{m}\}\subset\sym(\pg)^K$ as in the proof of Theorem \ref{N-stab} and denote by $\widetilde{\Le_\pg}$ the linear operator of $\la\cca\ra$ obtained by restricting and projecting $\Le_\pg$ on $\la\cca\ra$, where $\la\cca\ra$ is the subspace generated by $\cca$.

\begin{lemma}\label{LO-str}
\hspace{1cm}
\begin{enumerate}[{\rm (i)}]
\item \cite{ChnNknNkn} For all $1\leq i<k\leq m$,
$$
[iii]=\frac{(1-i)^2\dim{\fg}}{i(1+i)}, \qquad [iik]=\frac{\dim{\fg}}{k+k^2}.
$$
\item \cite{ChnNknNkn} $\rho=\frac{m+3}{4(m+1)}$.

\item The symmetric $m\times m$ matrix $[\widetilde{\Le_\pg}]_\cca$ is given by,
$$
  \left[\widetilde{\Le_\pg}\right]_{rs} = \left\{\begin{aligned}
  &\frac{-1}{r+r^2} \quad \text{ if } \ r>s, \\
  &\frac{-1}{s+s^2} \quad \text{ if } \ r<s, \\
  \end{aligned} \right. \qquad \qquad
  \left[\widetilde{\Le_\pg}\right]_{ss}  =\frac{s-1-s^2}{s+s^2}+\frac{m}{m+1}.
$$
\end{enumerate}
\end{lemma}

\begin{proof}
Let $\{ X_1,\ldots,X_{\dim{\fg}}\}$ be a $-\kil_\ggo$-orthonormal basis of $\fg$, as before we can compute the structural constants using that the set $\left\{\tfrac{1}{|\alpha_i|}\alpha_i(X_l)\right\}$ is a $-\Bb_\ggo$-orthonormal basis of $\ggo_{\alpha_i}$ and formula (\ref{SC}):
\begin{equation*}
  \left[iik\right]=  \sum_{lmr}\left<\left[\tfrac{\alpha_i(X_r)}{|\alpha_i|},\tfrac{\alpha_i(X_l)}{|\alpha_i|}\right],\tfrac{\alpha_k(X_m)}{|\alpha_k|}\right>^2 =  \left(\tfrac{\left<\alpha_i\cdot\alpha_i,\alpha_k\right>}{|\alpha_i|^2|\alpha_k|}\right)^2 \sum_{lmr}\left<[X_r,X_l],X_m\right>^2.
\end{equation*}
Since,
\begin{equation*}
\left<\alpha_i\cdot\alpha_i,\alpha_k\right>=\left\{
\begin{aligned}
                        i+i^2& \ \quad \text{ if } \ i<k  \\
                        i-i^3&  \ \quad \text{ if } \ i=k\\
  \end{aligned}  \right.,
\end{equation*}
the formulas stated in part (i) follow.

Parts (ii) and (iii) follow by setting $n=0$ in Lemma \ref{str}, (ii) (setting $\sum\limits_{j=0}^{n-1}\frac{1}{\Delta_j}=0$) and the proof of Theorem \ref{N-stab}, respectively.  Note that in this case we have $m$ irreducible factors of $\pg$.
\end{proof}


\begin{lemma}\label{spec}
The spectrum of $\widetilde{\Le_\pg}$ is given by $0$ and 
$$
a_i:=\frac{m}{m+1}-\frac{i}{i+2}, \qquad i=1,\dots,m-1.
$$
\end{lemma}

\begin{proof}
If we set $\eta_s:=\left[\widetilde{\Le_\pg}\right]_{ss}$, then by Lemma \ref{LO-str}, (iii), we have that the $m\times m$ matrix of $\widetilde{\lic_\pg}$ is given by
$$
\left[\widetilde{\Le_\pg}\right]_\mathcal{C}=\left[\begin{smallmatrix}
           \eta_1  & -\tfrac{1}{6} &-\tfrac{1}{12} & \cdots & &-\tfrac{1}{m+m^2}     \\
           -\tfrac{1}{6}  & \eta_2  &-\tfrac{1}{12} & \cdots & & -\tfrac{1}{m+m^2}   \\
             -\tfrac{1}{12}& -\tfrac{1}{12}   & \eta_3  & \cdots & &   \\
            &   &  & \ddots & &   \\
             &  & & &\eta_{m-1} &  -\tfrac{1}{m+m^2}\\
           -\tfrac{1}{m+m^2} & -\tfrac{1}{m+m^2} & -\tfrac{1}{m+m^2}   & \cdots  & -\tfrac{1}{m+m^2} &\eta_m
          \end{smallmatrix}\right].
$$
It is now easy to check that each vector
$$
(\underbrace{1,\ldots,1}_i,-i,\underbrace{0,\ldots,0}_{m-1-i})\in\RR^{m}, \qquad i=1,\dots,m-1,
$$
is an eigenvector of $\widetilde{\Le_\pg}$ with eigenvalue $a_i$ and its kernel is generated by $(1,\dots,1)$, concluding the proof (we are using here that $\sum_{j=1}^{i}\frac{1}{j+j^2}=\frac{i}{i+1}$). 
\end{proof}



\begin{theorem}\label{LO}
The standard metric $\gk$ on the Ledger-Obata space $M=F^{m+1}/\Delta^{m+1}F$ is $G$-unstable and has coindex $\geq m-2$ if $11\leq m$ and coindex $\geq m-1$ if $2\leq m\leq 10$.
\end{theorem}

\begin{proof}
The smallest non-zero eigenvalue of $\widetilde{\Le_\pg}$ is $a_{m-1}=\frac{1}{m+1}$.  Thus 
$$
\lambda_\pg\leq a_{m-1}=\tfrac{1}{m+1}<\tfrac{m+3}{2(m+1)}=2\rho,
$$ 
which implies that the standard metric on every Ledger-Obata space is G-unstable.  On the other hand, since $a_i<2\rho$ if and only if $\frac{2m-6}{m+5}<i$, we obtain the lower bounds for the coindex as stated, concluding the proof.
\end{proof}

\section{$G$-stable Einstein metrics with $G$ non-simple}\label{Gstab-sec}

In the light of the $G$-instability results obtained in \S\ref{nik-sec} and \S\ref{LO-sec}, it is natural to ask whether any $G$-invariant Einstein metric on a homogeneous space $G/K$ such that $G$ is not simple is $G$-unstable.  The aim of this section is to show that this is not true.  

Given a simple Lie group $H$ and a simple proper subgroup $K\subset H$, we consider the semisimple Lie group $G:=H\times K$, which acts transitively on $H$ by
$$
(\bar{h},k)\cdot h = \bar{h}hk^{-1},
$$
with isotropy subgroup at the identity given by $\Delta K$, the diagonal subgroup of $K$. This provides a presentation $M=G/\Delta K$ of the Lie group $M=H$ as a homogeneous space.

If $\hg=\kg\oplus\ag$ is the $\Bb_\hg$-orthogonal decomposition, then the $\Bb_\ggo$-orthogonal reductive decomposition of $M=G/\Delta K$ is given by,
$$
\ggo=\hg\oplus\kg=\Delta\kg\oplus\pg, \qquad \pg:=\ag\oplus \widetilde{\pg},
$$
where 
$$
\widetilde{\pg}:=\{X\in \ggo : X_\kg=-\tfrac{1}{c}X_\hg\} = \{(Z,-\tfrac{1}{c}Z):Z\in\kg\},
$$ 
(note that $\dim{\widetilde{\pg}}=\dim{\kg}$), $X_\hg$ and $X_\kg$ denote the projections of $X$ relative to $\ggo=\hg\oplus\kg$, respectively, and $\Bb_\kg=c\Bb_\hg|_{\kg\times\kg}$ (note that $0<c<1$).

It is easy to see that the differential at the origin of the diffeomorphism $\psi:G/\Delta K\rightarrow H$ determined by the above action is given by the $\Ad(K)$-equivariant map
$$
d \psi |_{o}: \pg \rightarrow \hg, \qquad d \psi |_{o} (X_\hg+X_\kg)=X_\hg-X_\kg, \qquad\forall X\in \pg,
$$
and its inverse $\varphi:=(d \psi |_{o})^{-1}: \hg \rightarrow \pg $ by
$$
 \varphi(Z)=\left\{\begin{aligned}
               & \qquad Z & \text{ if } Z\in \ag, \\
               & \left(\tfrac{c}{c+1}Z, (-\tfrac{1}{c}) \tfrac{c}{c+1}Z\right) & \text{if } Z\in \kg.
              \end{aligned}\right.
$$
Using that $\varphi$ is $\Ad(K)$-equivariant, one obtains the following diffeomorphism:
$$
\widetilde{\varphi}:\mathcal{M}^{H,K} \rightarrow \mathcal{M}^G, \qquad \widetilde{\varphi}(\bar{g}):= \bar{g}(\varphi^{-1} \cdot,\varphi^{-1} \cdot),
$$
where $\mathcal{M}^{H,K}$ is the manifold of all left-invariant metrics on $H$ which are in addition Ad($K$)-invariant and $\mathcal{M}^G$ is the manifold of all $G$-invariant metrics on $G/\Delta K$.  Note that $\widetilde{\varphi}$ is actually an isomorphism between the vectors spaces $\sym^2(\hg)^K$ and $\sym^2(\pg)^K$ of $\Ad(K)$-invariant symmetric $2$-tensors, respectively.

It is well known that the Killing metric $\bar{g}_B:=-\Bb_{\hg}$ on the simple Lie group $H$ is Einstein with $\unc$ as Einstein constant, thus the above diffeomorphism maps $\bar{g}_B$ to a $G$-invariant metric $g_0$ on $G/\Delta K$ which is also Einstein with $\ricci(g_0)=\unc g_0$.

\begin{proposition}\label{prod1}
For any simple Lie group $H\neq \SU(n), \Spe(n)$ and simple subgroup $K\subset H$, there exists an $(H\times K)$-invariant Einstein metric $g_0$ on the homogeneous space $M=H\times K/\Delta K$ which is $(H\times K)$-stable.
\end{proposition}

\begin{proof}
Since the corresponding scalar curvature functionals $\overline{\scalar}:\mathcal{M}^{H,K}\rightarrow \RR$ and $\scalar:\mathcal{M}^G\rightarrow \RR$ satisfy that $\overline{\scalar}=\scalar\circ\widetilde{\varphi}$, 
for all $T\in\sym^2(\hg)^K$,
$$
\overline{\scalar}''_g(T,T)=\left.\frac{d^2}{dt^2}\right|_0\overline{\scalar}(g+tT)=\left.\frac{d^2}{dt^2}\right|_0\scalar(\widetilde{\varphi}(g+tT))=\scalar''_{\widetilde{\varphi}(g)}(\widetilde{\varphi}(T),\widetilde{\varphi}(T)).
$$
In particular, $\overline{\scalar}''_g$ and $\scalar''_{\widetilde{\varphi}(g)}$ have the same signature as bilinear forms, but it is well known that $\bar{g}_B$ is $H$-stable if $H\neq \SU(n), \Spe(n)$ (see e.g.\ \cite[Proposition 5.1]{stab-tres}), so $g_0=\widetilde{\varphi}(\bar{g}_B)$ is $(H\times K)$-stable on $M=H\times K/\Delta K$.
\end{proof}

In what follows, we study the existence of $(H\times K)$-stable Einstein metrics on $M=H\times K/\Delta K$, including the case when $H$ is $\SU(n)$ or $\Spe(n)$.

We fix the standard metric $\gk:=-\Bb_{\ggo}|_{\pg} \in \mathcal{M}^G$ as a background metric.  Using \cite[Section 7]{stab-tres}, we obtain that the structural constants of the decomposition $\pg=\pg_1\oplus\pg_2$, where $\pg_1:=\ag$ and $\pg_2:=\tilde{\pg}$, are given by
$$
[111]=d_1-2(1-c)d_2 \qquad \qquad [112]=\frac{c(1-c)}{1+c}d_2 \qquad \qquad [222]=\frac{(c-1)^2}{c+1}d_2,
$$
where $d_1:=\dim{\ag}$ and $d_2:=\dim{\kg}$.  Since $\Bb_{\ggo}=\Bb_{\hg}+\Bb_{\kg}$, it is easy to check that
$$
g_0=x_1 \gk|_{\pg_1}+x_2\gk|_{\pg_2}, \qquad \mbox{where}\quad x_1=1, \quad x_2=\frac{c+1}{c}.
$$
One can verify that indeed $\ricci(g_0)=\unc g_0$ by using e.g.\ \cite[Section 2.5]{stab-dos}.

\begin{theorem}\label{prod2}
For any simple Lie group $H$ and simple subgroup $K\subset H$ such that the homogeneous space $H/K$ is isotropy irreducible, 
there exists an $(H\times K)$-stable Einstein metric $g_0$ on $M=H\times K/\Delta K$.
\end{theorem}

\begin{proof}
We first note that the $\Ad(K)$-representations $\pg_1$ and $\pg_2$ are irreducible and inequivalent (see \cite{DtrZll}).  It follows from \cite[Theorem 3.1]{stab-dos} that, for $g_0=\gk|_{\pg_1}+\frac{c+1}{c}\gk|_{\pg_2}$, the matrix of the Lichnerowicz Laplacian with respect to the orthonormal basis $\{\tfrac{1}{\sqrt{d_1}}I_1,\tfrac{1}{\sqrt{d_2}} I_{2}\}$ of $\text{sym}(\pg)^K$ is given by
$$
\left[\Le_{\pg}\right]=(1-c)\left[\begin{matrix}
           \tfrac{d_2}{d_1}  & -\sqrt{\tfrac{d_2}{d_1}}  \\
           -\sqrt{\tfrac{d_2}{d_1}} & 1 \\
             \end{matrix}\right].
$$
Its eigenvalues are $0$ and $\lambda_\pg=\frac{(d_1+d_2)(1-c)}{d_1}$, and by the bounds given in \cite[Theorem 11]{DtrZll} we can conclude that 
$\lambda_\pg> \unm=2\rho$ and so $g_0$ is $(H\times K)$-stable.
\end{proof}


\begin{remark}\label{Jensen-rem}
Under the hypothesis of the above theorem, $\dim{\mca^G_1}=1$ and there is a second Einstein metric found by Jensen \cite{Jns}, which is a normal metric on the homogeneous space $M=H\times K/\Delta K$ (see \cite{DtrZll}), given by 
$$
g_J=\gk|_{\pg_1}+\frac{c+1}{c}t_J\gk|_{\pg_2} = \widetilde{\vp}(\overline{g}_J),
$$
where $t_J:=\frac{d_1c}{(d_1+2d_2)(1-c)}$ and $\overline{g}_J$ is the left-invariant metric on $H$ defined by $\overline{g}_J:=-\kil_\hg|_\ag+t_J(-\kil_\hg)|_\kg$ (see \cite[\S7]{stab-tres} for a detailed treatment).  As far as we know, $g_J$ is the only known non-standard normal Einstein metric.  While $g_0$ is a local maximum of $\scalar:\mca^G_1\rightarrow\RR$ by Theorem \ref{prod2}, the Einstein metric $g_J$ was proved to be a local minimum in \cite[Proposition 7.2]{stab-tres}.  
\end{remark}

\end{document}